\documentclass[11pt]{amsart}

\usepackage{amsmath}
\usepackage{amscd}
\usepackage{amssymb}
\usepackage{amsthm}
\usepackage{latexsym}
\usepackage{mathabx}
\usepackage{mathtools}
\usepackage{hyperref}
\usepackage{stmaryrd}
\usepackage{tikz-cd}
\usepackage{enumitem}

\usepackage{dsfont}

\usepackage[arrow,curve,matrix,tips,2cell]{xy}
  \SelectTips{eu}{10} \UseTips
  \UseAllTwocells

\usepackage{calrsfs}
\usepackage[T1]{fontenc} 
\usepackage{textcomp}
\usepackage{times}
 \usepackage[scaled=0.92]{helvet}

\usepackage{geometry}
\newtheorem{theorem}{Theorem}[section]
\newtheorem{lemma}[theorem]{Lemma}
\newtheorem{corollary}[theorem]{Corollary}
\newtheorem{proposition}[theorem]{Proposition}

\newtheorem{theoremA}{Theorem}

\theoremstyle{definition}
\newtheorem{remark}[theorem]{Remark}
\newtheorem{definition}[theorem]{Definition}
\theoremstyle{plain}

\newcommand{\cP}{\mathcal P}
\newcommand{\cQ}{\mathcal Q}

\newcommand{\cS}{\mathcal S}

\def\Cz{\mathbb{C}}
\def\Fz{\mathbb{F}}

\def\Nz{\mathbb{N}}
\def\Qz{\mathbb{Q}}
\def\Rz{\mathbb{R}}

\def\Zz{\mathbb{Z}}

\newcommand{\mfp}{\mathfrak p}
\newcommand{\mfq}{\mathfrak q}

\newcommand{\ab}{\mbox{{\tiny \textup{ab}}}}

\newcommand{\Gal}{\mathrm{Gal}}

\newcommand{\Gkab}{\textup{G}^{\ab}_K}

\newcommand{\Hkab}{\widecheck{G}_{K}}
\newcommand{\Hlab}{\widecheck{G}_{K}}
\newcommand{\Hkabx}{\widecheck{G}_{K'}}
\newcommand{\Hlabx}{\widecheck{G}_{K'}}
\newcommand{\Frob}[1]{\textup{Frob}_{#1}}
\newcommand{\chipos}{\chi^{{+}}}

\newcommand{\chiram}[1]{\chi_{\textup{ram}, {#1}}}
\newcommand{\chiunram}[1]{\chi_{\textup{unr}, {#1}}}

\newcommand{\Lsp}[1]{L_p(A/K, #1, T)}
\newcommand{\Lspx}[1]{L_p(A'/K', #1, T)}

\newcommand{\ap}{\big(a_{\mfp}\big)}

\newcommand{\aphip}{\big(a'_{\phi(\mfp)}\big)}

\newcommand{\apf}{\big(a_{\mfp}\big)_{f/f_{\mfp}}}
\newcommand{\aphipf}{\big(a'_{\phi(\mfp)}\big)_{f/f_{\phi(\mfp)}}}

\newcommand{\aqf}{\big(a'_{\mfq}\big)_{f/f_{\mfq}}}

\newcommand{\apfi}{\big(a_{\mfp_i}\big)_{f/f_{\mfp_i}}}
\newcommand{\apfj}{\big(a_{\mfp_j}\big)_{f/f_{\mfp_j}}}
\newcommand{\aqfi}{\big(a'_{\mfq_i}\big)_{f/f_{\mfq_i}}}
\newcommand{\aqfj}{\big(a'_{\mfq_j}\big)_{f/f_{\mfq_j}}}

\newcommand{\cQf}{\cQ_{K, <f}}
\newcommand{\cQfx}{\cQ_{K', <f}}
\newcommand{\sumf}{\sum_{\mfp \in \cQ_{K, f}}}

\newcommand{\sumfx}{\sum_{\mfq \in \cQ_{K', f}}}
\newcommand{\sumfpx}{\sum_{\mfq \in \cQ_{K', f}^+}}

\newcommand{\DistTo}{\xrightarrow{
   \,\smash{\raisebox{-0.45ex}{\ensuremath{\scriptstyle\sim}}}\,}}

\newcommand\restr[2]{{% we make the whole thing an ordinary symbol
  \left.\kern-\nulldelimiterspace % automatically resize the bar with \right
  #1 % the function
  \vphantom{\big|} % pretend it's a little taller at normal size
  \right|_{#2} % this is the delimiter
  }}

\newcommand{\new}[1]{\textcolor{black}{#1}}

\begin{document}

\date{\today\ (version 1.0)} 
\title{$L$-series and isogenies of abelian varieties}
\author[H.~Smit]{Harry Smit}
\address{(HS) Mathematisch Instituut, Universiteit Utrecht, Postbus 80.010, 3508 TA Utrecht, Nederland}
\email{h.j.smit@uu.nl}

\subjclass[2010]{11G40, 11R37, 11R42, 14K02}
\keywords{\normalfont Class field theory, $L$-series, abelian varieties, isogenies}

\begin{abstract}
Faltings's isogeny theorem states that two abelian varieties are isogenous over a number field precisely when the characteristic polynomials of the reductions at almost all prime ideals of the number field agree. This implies that two abelian varieties over $\Qz$ with the same $L$-series are necessarily isogenous, but this is false over a general number field. Let $A$ and $A'$ be two abelian varieties, defined over number fields $K$ and $K'$ respectively. Our main result is that $A$ and $A'$ are isogenous after a suitable isomorphism between $K$ and $K'$ if and only if the Dirichlet character groups of $K$ and $K'$ are isomorphic and the $L$-series of $A$ and $A'$ twisted by the Dirichlet characters match.
\end{abstract}

\maketitle
\section{Introduction}
When two abelian varieties $A$ and $A'$ are isogenous over a number field $K$, then for any prime ideal (``prime'' from now on) $\mfp$ of $K$, the reductions of $A$ and $A'$ at $\mfp$ have the same characteristic polynomial. Faltings's isogeny theorem (\cite[Kor.\ 2]{faltings}) guarantees the converse: if the characteristic polynomials of the reductions are equal at every prime, then $A$ and $A'$ are isogenous.

In this paper we will mainly be concerned with the $L$-series $L(A/K, s)$ of an abelian variety $A/K$, an Euler product defined in terms of the characteristic polynomials. For an abelian variety $A$ defined over $\Qz$ one can read off the characteristic polynomial of the reduction at $p$ from $L(A/\Qz, s)$ by looking at the coefficients of $p^{-js}$ for $j \in \Nz$. Hence the theorem of Faltings implies that 
\[
L(A/\Qz, s) = L(A'/\Qz, s) \Longleftrightarrow A \text{ and } A' \text{ are isogenous over $\Qz$.}
\]
We make two observations about this reformulation of Faltings's result for abelian varieties over $\Qz$. First of all, the equivalence does not hold if one replaces $\Qz$ by an arbitrary number field $K$, even if $A$ is an elliptic curve: see Lemma~\ref{lemma:not determined} or \cite[Rmk.\ 3.4]{silverberg}. This phenomenon is very much akin to the fact that general number fields are not characterized by their zeta function (\cite[p.\ 671--672]{gassmann});  one can only extract information about the combined product of all Euler factors for primes lying above the same rational prime number. Secondly, in Faltings's theorem $A$ and $A'$ need to be defined over the same number field, whereas there is no natural reason to assume this when considering equalities of $L$-series. 

Our main result states that one can obtain an isogeny between abelian varieties, defined over possibly different number fields, when considering equalities of $L$-series twisted by Dirichlet characters (see Definition~\ref{definition:twisted_L_series}):

\begin{theoremA}\label{theoremA:abelian varieties are isogenous through L-series}
Let $K$ and $K'$ be number fields, and let $A$ and $A'$ be two abelian varieties defined over $K$ and $K'$ respectively. Let $l$ be a prime strictly larger than twice the dimension of $A$, and let $\Hkab[l]$ be the group of Dirichlet characters of the absolute Galois group $G_K$ of order $l$. Suppose there is a group isomorphism $\psi: \Hlab[l] \DistTo \Hlabx[l]$ such that
\[
L(A/K, \chi, s) = L(A'/K', \psi(\chi), s)
\]
for every $\chi \in \Hlab[l]$. Then there exists an isomorphism $\sigma: K \to K'$ such that $A^\sigma$ is isogenous to $A'$ over $K'$, and furthermore $\sigma$ induces $\psi$.
\end{theoremA}

Our second result is that in the case of elliptic curves one can also take $l = 2$.

\begin{theoremA} \label{theoremA: elliptic curves are isogenous through L-series}
Let $K$ and $K'$ be number fields, and let $E$ and $E'$ be elliptic curves defined over $K$ and $K'$ respectively. Suppose there exists a group isomorphism $\psi: \Hlab[2] \DistTo \Hlabx[2]$ such that
\[
L(E/K, \chi, s) = L(E'/K', \psi(\chi), s)
\]
for any $\chi \in \Hlab[2]$. \new{Then there exists an isomorphism $\sigma: K \DistTo K'$ such that $E^\sigma$ is isogenous to $E'$ over $K'$.}
\end{theoremA}

% $E$ does not have complex multiplication, or that the complex multiplication field of $E$ is contained in $K$.

The proofs of Theorems~\ref{theoremA:abelian varieties are isogenous through L-series} and~\ref{theoremA: elliptic curves are isogenous through L-series} consist of three parts. First, we establish a prime map between large subsets of the primes of $K$ and $K'$ that connects the values of $\chi$ and $\psi(\chi)$ and show that it is a bijection on a density one subset of the primes of $K$. This prime map is created through the use of characters with special properties, which exist due to the Grunwald-Wang theorem (see \cite[Ch.\ X, Thm.\ 5]{artintate}). For Theorem~\ref{theoremA: elliptic curves are isogenous through L-series}, the existence of this map relies on Serre's open image theorem for elliptic curves without complex multiplication (see \cite{serreopen}) and Deuring's criterion \cite{deuring} for CM elliptic curves. It follows from \cite[Thm.\ A \& B]{GDTHJS} that there exists a $\sigma: K \DistTo K'$, and finally we use Faltings's isogeny theorem  to conclude that the abelian varieties are isogenous after applying $\sigma$. 

\new{Even though Theorem~\ref{theoremA: elliptic curves are isogenous through L-series} holds for any elliptic curve, there still is a distinction to be made between those for which $\text{End}_{\overline{K}}(E) \subseteq K$ and those for which $\text{End}_{\overline{K}}(E) \nsubseteq K$. In the first case, all $\psi$ for which the conditions of Theorem~\ref{theoremA: elliptic curves are isogenous through L-series} hold are those induced by isomorphisms between $K$ and $K'$. In the second case additional isomorphisms $\psi$ may exist (see Remark~\ref{remark:extra_iso}).}

The article is structured as follows. In Section~\ref{section:preliminaries} we set up notation and introduce necessary definitions. We prove a large part of Theorem~\ref{theoremA: elliptic curves are isogenous through L-series} in Section~\ref{section:proof non-CM elliptic curves} before proving Theorem~\ref{theoremA:abelian varieties are isogenous through L-series} (partially) in Section~\ref{section:abelian varieties}, because many proofs of lemmas used for Theorem~\ref{theoremA:abelian varieties are isogenous through L-series} have a simpler version in the setting of Theorem~\ref{theoremA: elliptic curves are isogenous through L-series}. Then, in Section~\ref{section:prime bijections} we show that the results gathered in Sections~\ref{section:proof non-CM elliptic curves} and~\ref{section:abelian varieties} suffice to conclude Theorems~\ref{theoremA:abelian varieties are isogenous through L-series} and~\ref{theoremA: elliptic curves are isogenous through L-series}.

We end the introduction with some remarks and questions. A situation similar to Theorem~\ref{theoremA:abelian varieties are isogenous through L-series} occurs in the case of $L$-series of number fields. Here Ga{\ss}mann has shown (\cite[p.\ 671--672]{gassmann}) that two number fields with the same zeta function are not necessarily isomorphic. Theorem~\ref{theoremA:abelian varieties are isogenous through L-series} is an analogue of \cite[Thm.\ 3.1]{CdSLMS}: it shows that the existence of an $L$-series preserving isomorphism between two character groups necessitates that the underlying global fields are isomorphic. One wonders, taking into consideration \cite[Thm.\ 10.1]{CdSLMS}, whether or not two non-isogenous abelian varieties can always be distinguished by finitely many twisted $L$-series, or even a single one.

Lastly, in \cite{CLMS} a dynamical system is constructed that distinguishes non-isomorphic number fields, based on the abelianized Galois group. It would be interesting to generalize this to the case of abelian varieties.

\section{Preliminaries}\label{section:preliminaries}

We fix an algebraic closure $\overline{\Qz}$ of $\Qz$ throughout the entire paper. We denote number fields by $k$, $K$, $K'$ and $L$. We use $\mfp$ for the prime ideals (henceforth called ``primes'') of $K$ and $\mfq$ for the primes of $K'$.

The set of primes of a number field $K$ is denoted $\cP_{K}$, and the set of primes lying over a rational prime $p$ is denoted $\cP_{K, p}$. Given a prime $\mfp \in \cP_K$, we denote the norm of $\mfp$ by $N\mfp := \#\big(\mathcal{O}_K/\mfp\big) = p^{f_{\mfp}}$, where $f_{\mfp}$ is the inertia degree of $\mfp$. We denote the local field at a prime $\mfp$ of $K$ by $K_{\mfp}$, and use $\Fz_{\mfp}$ for the residue field.

For a number field $K$, let $G_K$ be the absolute Galois group, $K^\textup{ab}$ be the composite of all abelian extensions of $K$, and denote by $\Gkab$ its Galois group over $K$. The dual, denoted $\Hlab$, is the group of all Dirichlet characters (i.e. continuous homomorphisms) $G_K \to \Cz^\times$. For any prime $l$ we denote by $\Hlab[l]$ the subgroup of $\Hlab$ generated by characters of order $l$. Denote by $1_K \in \Hlab$ the trivial character.

We use the following convenient notation from combinatorics: if $P(T)$ is a polynomial, denote by $[T^n]P(T)$ the coefficient of $T^n$ of $P(T)$.

\subsection{Dirichlet characters} \leavevmode \\
Associated to every Dirichlet character $\chi \in \Hkab$ is a unique finite cyclic extension $K_{\chi}/K$ of degree equal to the order of $\chi$ such that $\chi$ factors through $\Gal(K_{\chi}/K)$. Let $\mfp$ be a prime of $K$ unramified in $K_{\chi}/K$ and let $\Frob{\mfp}$ be the Frobenius element in $\Gal(K_{\chi}/K)$. We set $\chi(\mfp) := \chi(\Frob{\mfp})$. If $\mfp$ is a prime of $K$ that ramifies in $K_{\chi}/K$, we set $\chi(\mfp) = 0$.

Throughout the paper we will be concerned with the existence of characters with certain properties (mainly with prescribed values at specific primes). For our purposes, the question of whether such characters exist is answered by the Grunwald-Wang theorem \cite[Ch.\ X, Thm.\ 5]{artintate}. It states the following: let $\mfp_1, \dots, \mfp_n$ be primes of $K$, and let $K_{\mfp_i}$ be the localization of $K$ at $\mfp_i$. For any $1 \leq i \leq n$, let $m_i$ be an integer and let $\chi_i \in \widecheck{G}_{K_{\mfp_i}}[m_i]$. Then, aside from a special case that occurs only when $m := \text{lcm}(m_1, \dots, m_n)$ is divisible by $4$, there exists a character $\chi \in \Hkab[m]$ such that $\chi(\mfp_i) = \chi_i(\mfp_i)$. 

The following two lemmas are useful corollaries of the Grunwald-Wang theorem.

\begin{lemma}\label{lemma:grunwald-wang in any order}
Let $K$ be a number field and $S = \{\mfp_1, \dots, \mfp_s\}$ a finite set of primes of $K$. Let $l$ be a prime number, and let $\zeta_1, \zeta_2, \dots, \zeta_s$ be $l^\text{th}$ roots of unity. Then there exists a character $\chi \in \Hlab[l]$ such that $\chi(\mfp_i) = \zeta_i$ for all $1 \leq i \leq s$. 
\end{lemma}

\begin{proof}
\cite[Lemma 2.1]{GDTHJS}.
\end{proof}

\begin{lemma}\label{lemma:grunwald-wang in order 2}
Let $K$ be a number field and $S = \{\mfp_1, \dots, \mfp_s\}$ a finite set of primes of $K$. Let $\zeta_1, \dots, \zeta_s$ be elements of $\{-1, 0, 1\}$. Then there exists a quadratic character $\chi$ such that $\chi(\mfp_i) = c_i$. 
\end{lemma}

\begin{proof}
Any of the local fields $K_{\mfp_i}$ has a quadratic ramified Galois extension, which is obtained by adjoining the square root of a uniformizer. If $\zeta_i = 0$ we choose $\chi_i$ to be the character associated to such an extension.

Every $K_{\mfp_i}$ also has an unramified Galois extension of degree $2$, obtained by adjoining a $(N\mfp_i^2 - 1)^\text{th}$ root of unity. If $\zeta_i = \pm 1$, let $\chi_i$ be a character that factorises through the Galois group of such an extension of $K_{\mfp_i}$, for which $\chi_i(\mfp_i) = \zeta_i$. The Grunwald-Wang theorem now guarantees that there is a $\chi \in \Hlab[2]$ such that $\chi(\mfp_i) = \chi_i(\mfp_i) = \zeta_i$. 
\end{proof}

\subsection{$L$-series of abelian varieties}\leavevmode \\
\noindent Let $A$ be an abelian variety of dimension $d$ defined over $K$, and let $l$ be a prime number. Let $T_lA = \varprojlim A[l^n]$ be the Tate module, and define $V_lA = T_lA \otimes_{\Zz_l} \Qz_l$, a $2d$-dimensional vector space over $\Qz_l$. 

Let $\mfp$ be any prime of $K$ coprime to $l$. Denote by $\rho_{A,l}: \Gal(\overline{K_{\mfp}}/K_{\mfp}) \to \text{Hom}(V_lA, \Qz_l)$ the $l$-adic representation of $\Gal(\overline{K_{\mfp}}/K_{\mfp})$, and fix an embedding $\overline{\Qz_l} \hookrightarrow \Cz$ to make the representation complex. Let $\Frob{\mfp} \in \Gal(\overline{K_{\mfp}}/K_{\mfp})$ be any lift of the geometric (i.e. inverse of the) Frobenius element in $\Gal(K^{\textup{unr}}_{\mfp}/K_{\mfp})$.

\begin{definition}\label{definition:twisted_L_series}
Let $A/K$ be an abelian variety, $\mfp$ a prime of $K$ and $l$ a prime unequal to the characteristic of $\mfp$. Denote by $I_{\mfp} \subset G_K$ the inertia group of $\mfp$ and let $\chi \in \Hlab$ be a Dirichlet character. We define the local $L$-series of $A/K$ at $\mfp$ twisted by $\chi$ by
\[
L_\mfp(A/K, \chi, T) = \det\big(1 - T^{f_{\mfp}} (\chi \otimes \rho_{A, l})(\Frob{\mfp}) \mid \big(\Cz \otimes \text{Hom}(V_lA, \Qz_l)\big)^{I_{\mfp}}\big).
\]
This definition is slightly non-standard; usually $T^{f_{\mfp}}$ is replaced by $T$. However, for our purposes it is more convenient to define the local $L$-series like this, see Definition~\ref{definition:factor at p}.

If $A$ has good reduction at $\mfp$, then we can write this polynomial as
\[
L_\mfp(A/K, \chi, T) = 1 + \chi(\mfp)\ap_1 T^{f_{\mfp}} + \chi(\mfp)^2\ap_2 T^{2f_{\mfp}} + \dots + \chi(\mfp)^{2d}\ap_{2d} T^{2df_{\mfp}}.
\]
We call the $\ap_i$ the \emph{local coefficients}. The degree of this polynomial equals $2d f_{\mfp}$, and all the roots have modulus $p^{-f_{\mfp}/2}$, hence we have $\ap_{2d} = p^{d f_{\mfp}}$, see \cite{weil}.

The global twisted $L$-series, denoted $L(A/K, \chi, s)$, is defined as
\[
L(A/K, \chi, s) = \prod_{p \in \cP_{\Qz}} \prod_{\mfp \in \cP_{K, p}} L_{\mfp}(A/K, \chi, p^{-s})^{-1}.
\]
\end{definition}

$L$-series can be written additively:
\[
L(A/K, \chi, s) = \sum_{n \in \mathbb{N}} c_n n^{-s}
\]

The $(c_n)_{n \in \Nz}$ uniquely determine $L(A/K, \chi, s)$ and vice versa (see \cite[Ch.\ 2, \S 2.2, Cor.\ 4]{serrecourse}). The equality of coefficients allows us to extract information about the product of the local factors, as made precise by the following definition and lemma:

\begin{definition}\label{definition:factor at p}
We define the factor at $p$, denoted $L_p(A/K, \chi, T)$ to be the product
\[
L_p(A/K, \chi, T) := \prod_{\mfp \mid p} L_{\mfp}(A/K, \chi, T).
\]
\end{definition}

\begin{lemma}\label{lemma:K=K' iff Lp=Lp'}
We have an equality of $L$-series $L(A/K, \chi, s) = L(A'/ K', \chi', s)$ if and only if $L_p(A/K, \chi, s) = L_p(A'/ K', \chi', s)$ for every rational prime $p$.
\end{lemma}

\begin{proof}
\cite[Lemma 5.6]{GDTHJS} proves this for Dirichlet $L$-series of number fields. The same proof holds for $L$-series of abelian varieties.
\end{proof}

If $A/K$ is an abelian variety, and $\sigma$ is an automorphism of $K$, then for any $\mfp \in \cP_{K}$ we have
\[
L_{\mfp}(A/K, T) = L_{\sigma(\mfp)}(A^\sigma/K, T).
\]
As $\sigma$ permutes the primes lying over any rational prime $p$, it follows that
\[
L_{p}(A/K, T) = L_{p}(A^\sigma/K, T)
\]
for any prime number $p$. As a result, $L(A/K, s) = L(A^\sigma/K, s)$. However, $A$ and $A^\sigma$ need not be isogenous: take for example the elliptic curve over $\Qz(i)$ defined by $E:y^2 = x^3 + ix$, and let $\mfp_1 = (2 + i)$ and $\mfp_2 = (2-i)$. Let $\sigma$ be the conjugation automorphism, and note $\sigma (\mfp_1) = \mfp_2$. As explained above, $L(E/K, s) = L(E^\sigma/K, s)$, but $L_{\mfp_2}(E/K, s) \neq L_{\mfp_1}(E/K, s) = L_{\mfp_2}(E^\sigma/K, s)$ as the reduction of $E$ to $\Fz_{\mfp_1}$ has ten points,  but the reduction of $E$ to $\Fz_{\mfp_2}$ only has two. It follows that $E$ and $E^\sigma$ are not isogenous over $\Qz(i)$, as they have different characteristic polynomials at $\mfp_2$. This proves the following lemma.

\begin{lemma}\label{lemma:not determined}
The isogeny type of an abelian variety over an arbitrary number field is not characterized by its $L$-series. \qed
\end{lemma}

\section{Most elliptic curves are characterized by their quadratic twists}\label{section:proof non-CM elliptic curves}

In this section we consider the $L$-series of elliptic curves twisted with quadratic characters. The main difference between quadratic twists and twists of higher order is the fact that any quadratic character $\chi \in \Hlab[2]$ unramified at $\mfp$ satisfies $\chi(\mfp)^2 = 1$. Hence the local $L$-series at $\mfp$ is of the form
\[
L_{\mfp}(E/K, \chi, T) = 1 + \chi(\mfp)a_{\mfp} T^{f_{\mfp}} + p^{f_{\mfp}} T^{2 f_{\mfp}}.
\]
(We use $a_{\mfp}$ as shorthand for $\ap_1$ in this section).
In particular, if $a_{\mfp} = 0$, this does not depend on $\chi$. This troubles the case where a significant number of the $a_{\mfp}$ equal zero, see Remark~\ref{remark:extra_iso}.

In this section we do not yet prove Theorem~\ref{theoremA: elliptic curves are isogenous through L-series} completely. Rather, we show that the conditions of the theorem imply the existence of an injective norm-preserving map of primes with certain properties.

 Let $S$ be the set of odd rational primes $p$ for which the following holds:
\begin{itemize}
\item $p$ is unramified in both $K$ and $K'$, and
\item $E$ and $E'$ have good reduction at all primes lying over $p$ of $K$ and $K'$ respectively.
\end{itemize}
Note that these conditions exclude finitely many primes (see \cite{serretateNOS}). Furthermore, define $\cS \subseteq \cP_{K}$ as be the set of primes of $K$ that lie over a prime in $S$ and define $\cS' \subseteq \cP_{K'}$ similarly.

\begin{theorem} \label{theorem: from L-functions of elliptic curves to prime bijection}
Let $K$ and $K'$ be number fields, and let $E$ and $E'$ be elliptic curves defined over $K$ and $K'$ respectively. Suppose there exists an isomorphism $\psi: \Hlab[2] \to \Hlabx[2]$ such that
\[
L(E/K, \chi, s) = L(E'/K', \psi(\chi), s)
\]
for any $\chi \in \Hlab[2]$. Then there exists a norm-preserving bijection of primes $\phi: \cS \to \cS'$ such that $a_{\mfp} = a'_{\phi(\mfp)}$ for any $\mfp \in \cP_K$ and 
\[
\chi(\mfp) = \psi(\chi)(\phi(\mfp))
\]
for any $\chi \in \Hlab[2]$ and $\mfp \in \cP_K$ such that $a_{\mfp} \neq 0$.
\end{theorem}

The remainder of this section is spent on the proof of this theorem. From now on, assume that the conditions of Theorem~\ref{theorem: from L-functions of elliptic curves to prime bijection} hold.

\begin{lemma}\label{lemma:equal degree}
We have $[K:\Qz] = [K': \Qz]$. 
\end{lemma}

\begin{proof}
For any $p \in S$, the degree of $L_p(E/K, 1_K, T)$ as a polynomial in $T$ equals $2[K:\Qz]$. Similarly, the degree of $L_p(E'/K', 1_{K'}, T)$ equals $2[K':\Qz]$. The result follows as $\psi(1_K) = \psi(1_{K'})$.
\end{proof}

 The bijection of primes of Theorem~\ref{theorem: from L-functions of elliptic curves to prime bijection} is created one rational prime at a time; for the remainder of this section, fix a prime number $p \in S$. We create characters $\chi$ that have special behaviour on a single prime lying over $p$, and prove that $\psi(\chi)$ must have similar properties, due to $L(E/K, \chi, s) = L(E'/K', \psi(\chi), s)$.

\begin{definition}\label{definition:ramified character of order 2}
Let $\chiram{\mfp} \in \Hlab[2]$ be a character that satisfies the following for any $\tilde{\mfp} \in \cP_{K, p}$:
\[
\chiram{\mfp}(\tilde{\mfp}) = 
\begin{cases}
0 & \text{ if } \tilde{\mfp} = \mfp;\\
1 & \text{ otherwise.}
\end{cases}
\]
Such a character exists by Lemma~\ref{lemma:grunwald-wang in order 2}.
\end{definition}

\begin{lemma}\label{lemma:same norm, same ramification}
There exists a norm-preserving bijection $\phi: \cP_{K,p} \to \cP_{K',p}$ such that for any prime $\mfp \in \cP_{K, p}$ and any $\chi \in \Hlab[2]$ we have that $\chi$ is ramified at $\mfp$ if and only if $\psi(\chi)$ is ramified at $\phi(\mfp)$. 
\end{lemma}

\begin{proof}
We proceed by induction. Denote by $\mfp_1, \dots, \mfp_n$ the primes of $K$ lying over $p$, sorted by norm in increasing order. \\

\noindent \textbf{Induction hypothesis.} Suppose for some $i \geq 1$ we have the following.
\begin{itemize}
\item For any $j < i$ we have a prime $\mfq_j \in \cP_{K', p}$ with the same norm as $\mfp_j$.
\item For any $j < i$ the character $\chi$ is ramified at $\mfp_j$ if and only if $\psi(\chi)$ is ramified at $\mfq_j$.
\item We have $f_{\mfq_1} \leq f_{\mfq_2} \leq \dots \leq f_{\mfq_{i-1}} \leq \min \{ f_{\mfq}: \mfq \in \cP_{K', p} \backslash\{\mfq_1, \dots, \mfq_{i-1}\}\}$.
\end{itemize}
This statement is empty for $i = 1$. \\

\noindent \textbf{Induction step.}
Without loss of generality, assume that (one can swap the roles of $K$ and $K'$ if necessary)
\[
f_{\mfp_i} \leq \min \{ f_{\mfq}: \mfq \in \cP_{K', p}\text{ such that }\mfq \neq \mfq_j\, \forall j < i\}.
\]

The degree of $L_p(E/K, \chiram{\mfp_i}, T)$ as a polynomial in $T$ equals $2[K:\Qz] - 2f_{\mfp_i}$, as $\chiram{\mfp_i}$ is only ramified at $\mfp_i$. Hence the degree of $L_p(E/K, \psi(\chiram{\mfp_i}), T)$ must also equal $2[K:\Qz] - 2f_{\mfp_i}$, which equals $2[K':\Qz] - 2f_{\mfp_i}$ by Lemma~\ref{lemma:equal degree}. We know that $\psi(\chiram{\mfp_i})$ is unramified at the primes $\mfq_1, \dots, \mfq_{i-1}$ by the induction hypothesis. Furthermore,
\[
\deg L_p(E/K, \psi(\chiram{\mfp_i}), T) = 2[K':\Qz] - 2 \sum_{\mfq \mid p:\; \psi(\mfq) = 0} f_{\mfq}.
\]
By assumption $f_{\mfp_i} \leq f_{\mfq}$ for all $\mfq \neq \mfq_j$ for any $j < i$, hence the equality 
\[
2[K':\Qz] - 2 \sum_{\mfq \mid p: \psi(\mfq) = 0} f_{\mfq} = 2[K:\Qz] - 2f_{\mfp_i}
\]
can only hold if $\psi(\chiram{\mfp_i})$ is ramified at a single prime in $\cP_{K', p}$, which we call $\mfq_i$. Now $f_{\mfp_i} = f_{\mfq_i}$ follows from $\deg L_p(E/K, \psi(\chiram{\mfp_i}), T) = 2[K':\Qz] - 2 f_{\mfq_i}$. 

As the character $\psi(\chiram{\mfp_i})$ ramifies at $\mfq_i$, but not at $\mfq_j$ for all $j< i$, we have $\mfq_i \neq \mfq_j$. \\

\noindent We complete the proof of Lemma~\ref{lemma:same norm, same ramification} by checking that $\chi$ is ramified at $\mfp_i$ if and only if $\psi(\chi)$ is ramified at $\mfq_i$. We argue by contradiction: assume without loss of generality that $\chi$ is a character that is unramified at $\mfp_i$, but $\psi(\chi)$ is ramified at $\mfq_i$. The character $\chi \chiram{\mfp_i}$ has the same value as $\chi$ everywhere except at $\mfp_i$, where $\chi$ is unramified but $\chi \chiram{\mfp_i}$ is ramified. In particular,
\begin{align*}
\deg L_p(E/K, \chi, T) &= \deg L_p(E/K, \chi \chiram{\mfp_i}, T) + 2 f_{\mfp_i} \\
					   &> \deg L_p(E/K, \chi \chiram{\mfp_i}, T).
\end{align*}
The character $\psi(\chi \chiram{\mfp_i}) = \psi(\chi)\psi(\chiram{\mfp_i})$ has the same value at every $\mfq \in \cP_{K', p}$ as $\psi(\chi)$ except (possibly) at $\mfq_i$. At $\mfq_i$ both $\psi(\chi)$ and $\psi(\chiram{\mfp_i})$ are ramified: we claim that $\psi(\chi)\psi(\chiram{\mfp_i})$ is unramified there. Indeed, suppose $\chi_1, \chi_2 \in \Hlab[2]$ both ramify at a prime $\mfp$. The extension $K_{\chi_i}/K$ has a primitive element of the form $\sqrt{d_i}$, with $d_i \in K^\times / (K^\times)^2$. The extension $K_{\chi_1 \chi_2}$ is obtained by adjoining $\sqrt{d_1d_2}$ to $K$. Because $\chi_1$ and $\chi_2$ ramify at $\mfp$, both $v_{\mfp}(d_1)$ and $v_{\mfp}(d_2)$ are odd, hence $v_{\mfp}(d_1d_2)$ is even, i.e. $\mfp$ does not ramify in $K_{\chi_1\chi_2}/K$. Therefore
\begin{align*}
\deg L_p(E'/K', \psi(\chi), T) &= \deg L_p(E'/K', \psi(\chi \chiram{\mfp_i}), T) - 2 f_{\mfq_i} \\
&< \deg L_p(E'/K', \psi(\chi \chiram{\mfp_i}), T).
\end{align*}
These two inequalities contradict the equalities 
\[
L_p(E/K, \chi, T) = L_p(E'/K', \psi(\chi), T)
\]
and 
\[
L_p(E/K, \chi \chiram{\mfp_i}, T) = L_p(E'/K', \psi(\chi \chiram{\mfp_i}), T).
\]
The result now follows by setting $\phi(\mfp_i) := \mfq_i$. We have shown that $f_{\mfp_i} = f_{\mfq_i}$ for all $i$, and $\mfq_i \neq \mfq_j$ for all $i \neq j$, hence $\phi$ is an injective norm-preserving map. It is a bijection by symmetry of $K$ and $K'$.
\end{proof}

\begin{remark}
As we showed that $\chi$ is ramified at $\mfp$ if and only if $\psi(\chi)$ is ramified at $\phi(\mfp)$, it follows that $\phi$ is independent of the choices of $\chiram{\mfp_i}$.
\end{remark}

Finally, we show that $\phi$ has the desired properties, namely $a_{\mfp} = a'_{\phi(\mfp)}$ and $\chi(\mfp) = \psi(\chi)(\phi(\mfp))$. The following definition and lemma allow us to create a connection between the values of $\chi$ and $\psi(\chi)$. 

\begin{definition}\label{definition:chipos order 2}
Define $\chipos \in \Hlab[2]$ to be any character that maximises the value of the map 
\begin{align*}
\Hlab[2] &\to \Rz \\
\chi &\mapsto L_p(E/K, \chi, 1).
\end{align*} 
\end{definition}

\begin{lemma} \label{lemma: properties of chipos and psichipos}
The character $\chipos$ has the following properties.
\begin{enumerate}
\item For any $\mfp \mid p$, we have 
\[
\chipos(\mfp) = 
\begin{cases}
1 & \text{ if } a_\mfp > 0;\\
-1 & \text{ if } a_\mfp < 0.
\end{cases}
\]
\item $\psi(\chipos)$ maximises the map $\chi' \mapsto L(E'/K', \chi', 1)$, hence
\[
\psi(\chipos)(\mfq) = 
\begin{cases}
1 & \text{ if } a'_\mfq > 0;\\
-1 & \text{ if } a'_\mfq < 0.
\end{cases}
\]
\end{enumerate}
\end{lemma}

\begin{proof} 
For any $\chi \in \Hlab[2]$ we have
\[
L_p(E/K, \chi, 1) = \prod_{\mfp \mid p} L_{\mfp}(E/K, \chi, 1) = \prod_{\mfp \mid p} (1 + \chi(\mfp)a_{\mfp} + p^{f_{\mfp}}).
\]
By the Hasse bound, $a_{\mfp} \leq 2p^{f_{\mfp} / 2} < p^{f_{\mfp}} + 1$, hence each term of the product is positive. It is therefore maximised by maximising each individual term, which is done precisely if
\[
\chi(\mfp) = 
\begin{cases}
1 & \text{ if } a_\mfp > 0;\\
-1 & \text{ if } a_\mfp < 0.
\end{cases}
\]

Let $\chi' \in \Hlabx[l]$. By maximality of $\chipos$, we have 
\[
L_p(E/K, \chipos, 1) \geq L_p(E/K, \psi^{-1}(\chi'), 1).
\]
As $L_p(E/K, \chi, T) = L_p(E'/K', \psi(\chi), T)$ for any $\chi \in \Hlab[2]$, this is equivalent to
\[
L_p(E'/K', \psi(\chipos), 1) \geq L_p(E'/K', \chi', 1).
\]
Hence $\psi(\chipos)$ maximises the map $\chi' \mapsto L(E'/K', \chi', 1)$. Using the first part of this lemma (applied to $K'$), it has the stated properties.
\end{proof}

\begin{lemma}\label{lemma:a=a'} 
For any $\mfp \in \cP_{K, p}$ we have $a_{\mfp} = a'_{\phi(\mfp)}$. 
\end{lemma}

\begin{proof}
Let $\chiunram{\mfp} \in \Hlab[2]$ be any character such that $\chiunram{\mfp}$ has value $1$ at $\mfp$ and is ramified at all other primes of $K$ lying over $\mfp$. Then 
\[
L_p(E/K, \chiunram{\mfp}, 1) = 1 + a_{\mfp} + p^{f_{\mfp}}.
\]
We know by Lemma~\ref{lemma:same norm, same ramification} that $\psi(\chiunram{\mfp})$ is ramified at all primes lying over $p$ except $\phi(\mfp)$. Hence
\[
L_p(E'/K', \psi(\chiunram{\mfp}), 1) = 1 + \psi(\chiunram{\mfp})(\phi(\mfp))a'_{\phi(\mfp)} + p^{f_{\phi(\mfp)}}.
\]
As $\phi$ is norm-preserving, we find $a_{\mfp} = \psi(\chiunram{\mfp})(\phi(\mfp))a'_{\phi(\mfp)}$. 
In particular it follows that $|a_{\mfp}| = |a'_{\phi(\mfp)}|$, and thus $a_{\mfp} = 0$ if and only if $a'_{\phi(\mfp)} = 0$. 

We now argue by contradiction. Suppose we have $a_{\mfp} \neq a'_{\phi(\mfp)}$. Then $a_{\mfp} \neq 0$ and $a_{\mfp} = - a'_{\phi(\mfp)}$, and therefore $\psi(\chiunram{\mfp})(\phi(\mfp)) = -1$. Without loss of generality, assume $a_{\mfp} > 0$. The character $\chi \chipos$ is ramified at all primes lying over $p$ except $\mfp$, hence
\[
L_p(E/K, \chiunram{\mfp} \chipos, 1) = 1 + \chi(\mfp) \chipos(\mfp) a_{\mfp} + p^{f_{\mfp}} =  1 + a_{\mfp} + p^{f_{\mfp}}.
\]
Similarly,
\[
L_p(E'/K', \psi(\chiunram{\mfp} \chipos), 1) = 1 + \psi(\chiunram{\mfp} \chipos)(\phi(\mfp)) a'_{\phi(\mfp)} + p^{f_{\phi(\mfp)}} =  1 + a'_{\phi(\mfp)} + p^{f_{\phi(\mfp)}}
\]
as $a'_{\phi(\mfp)} < 0$ implies $\psi(\chipos)(\phi(\mfp)) = -1$. However, then $a_{\mfp} = a'_{\phi(\mfp)}$; a contradiction. We conclude that $a_{\mfp} = a'_{\phi(\mfp)}$ for all $\mfp$ (hence $\psi(\chiunram{\mfp})(\phi(\mfp)) = 1$).
\end{proof}

\begin{lemma}\label{lemma: chip = psichiphip}
For any $\mfp \in \cP_{K, p}$ such that $a_{\mfp} \neq 0$ and for any $\chi \in \Hlab[2]$ we have $\chi(\mfp) = \psi(\chi)(\phi(\mfp))$.
\end{lemma}

\begin{proof}
We first prove this for all characters that are unramified at all primes lying over $p$. Let $\tilde{\chi}$ be such a character, and let $\mfp \in \cP_{K, p}$ be any prime lying over $p$ such that $a_{\mfp} \neq 0$. As in the previous lemma, let $\chiunram{\mfp} \in \Hlab[2]$ be any character such that $\chi(\mfp) = 1$ and that $\chiunram{\mfp}$ is ramified at all other primes of $K$ lying over $\mfp$. Then 
\[
L_p(E/K, \tilde{\chi} \chiunram{\mfp}, 1) = 1 + \tilde{\chi}(\mfp) a_{\mfp} + p^{f_{\mfp}}.
\]
In the proof of Lemma~\ref{lemma:a=a'} we have seen that $\psi(\chi)(\phi(\mfp)) = 1$. Hence
\[
L_p(E'/K', \psi(\tilde{\chi} \chiunram{\mfp}), 1) = 1 + \psi(\tilde{\chi})(\phi(\mfp)) a'_{\phi(\mfp)} + p^{f_{\phi(\mfp)}}.
\]
As a result, $\tilde{\chi}(\mfp) a_{\mfp} = \psi(\tilde{\chi})(\phi(\mfp)) a'_{\phi(\mfp)}$, hence by Lemma~\ref{lemma:a=a'} and the assumption that $a_{\mfp} \neq 0$ we find $\tilde{\chi}(\mfp) = \psi(\tilde{\chi})(\phi(\mfp))$. 

Now let $\chi \in \Hlab[2]$ be any character. We show that $\chi(\mfp) = \psi(\chi)(\phi(\mfp))$. Let $\chi_{\mfp}\in \Hlab[2]$ be any character that has value $-1$ at $\mfp$, and value $1$ at all other primes lying over $p$. It follows that
\[
[T^{f_{\mfp}}] \big(L_p(E/K, \chi, T) - L_p(E'/K', \chi \chi_{\mfp}, T) \big) = 2\chi(\mfp)a_{\mfp}.
\]
As $\chi_{\mfp}$ is unramified at all primes lying over $p$, the first part of the proof implies that $\psi(\chi_\mfp)$ has value $-1$ on $\phi(\mfp)$ and value $1$ on all other primes of $K'$ lying over $p$. Hence
\[
[T^{f_{\mfp}}] \big(L_p(E'/K', \psi(\chi), T) - L_p(E'/K', \psi(\chi \chi_{\mfp}), T)\big) = 2\psi(\chi)(\phi(\mfp))a'_{\phi(\mfp)},
\]
and thus $2\chi(\mfp)a_{\mfp} = 2\psi(\chi)(\phi(\mfp))a'_{\phi(\mfp)}$. In particular, if $a_{\mfp} \neq 0$, then $\chi(\mfp) = \psi(\chi)(\phi(\mfp))$.
\end{proof}

This concludes the proof of Theorem~\ref{theorem: from L-functions of elliptic curves to prime bijection}.

\section{Abelian varieties are determined by their twisted $L$-series of sufficiently high order}\label{section:abelian varieties}
The idea of the proof of Theorem~\ref{theoremA:abelian varieties are isogenous through L-series} is similar to that of Theorem~\ref{theoremA: elliptic curves are isogenous through L-series}; we create an injective norm-preserving map of primes $\phi: \cS \to \cP_{K'}$ for some density one subset $\cS \subseteq \cP_{K}$ with additional properties. However, we now consider characters of a general prime order $l$, which complicates matters:
\begin{itemize}
\item In general there do not exist ramified extensions of given degree $l > 2$ over a given prime $\mfp$, hence it is often impossible to create characters of order $l$ such that $\chi(\mfp) = 0$, which were vital in the proof of Theorem~\ref{theorem: from L-functions of elliptic curves to prime bijection}, see Definition~\ref{definition:ramified character of order 2}. This lengthens the proof significantly and necessitates some heavier calculations.
\item A character of order two admits only real values; a character of higher order does not. This makes it unfeasible to maximise the map $\chi \mapsto L_p(A/K, \chi, 1)$ as we did in Definition~\ref{definition:chipos order 2}. We work around this by using an inductive argument and a (slightly) different function to maximise (see Definition~\ref{definition:the map c}).
\end{itemize}

Let $S$ be the set of odd rational primes $p$ for which:
\begin{itemize}
\item $p$ is unramified in both $K$ and $K'$, and
\item $A$ and $A'$ have good reduction at all primes lying over $p$ of $K$ and $K'$ respectively.
\end{itemize}
Note that $S$ contains all but finitely many primes, as both $A$ and $A'$ have only finitely many primes of bad reduction (see \cite{serretateNOS}). Define $\cS \subseteq \cP_K$ as the set of primes of $K$ that lie over rational primes in $S$, and definite $\cS' \subseteq \cP_{K'}$ similarly.
 
As in Section~\ref{section:proof non-CM elliptic curves}, we do not prove Theorem~\ref{theoremA:abelian varieties are isogenous through L-series} completely in this section, but rather we show the following theorem holds.

\begin{theorem}\label{theorem:local factors are equal}
Let $K$, $K'$ be number fields, $A/K$ and $A'/K'$ abelian varieties of dimension $d$ and $d'$ respectively. Let $l > 2d$ be prime. Suppose there exists an isomorphism $\psi: \Hlab[l] \DistTo \Hlabx[l]$ such that 
\[
L(A/K, \chi, s) = L(A'/K', \psi(\chi), s)
\]
for all $\chi \in \Hlab[l]$. There exists a norm-preserving bijection of primes $\phi: \cS \to \cS'$ such that 
\[
L_\mfp(A/K, \chi, T) = L_{\phi(\mfp)}(A'/K', \psi(\chi), T)
\]
and $\mfp \in \cS$ for any $\chi \in \Hlab[l]$.
\end{theorem}

\begin{lemma}\label{lemma: same dimension}
Suppose there exists a bijection $\psi: \Hlab[l] \to \Hlabx[l]$ that maps the trivial character to the trivial character such that 
\[
L(A/K, \chi, s) = L(A'/K', \psi(\chi), s)
\]
then the dimension $d$ of $A$ is the same as the dimension $d'$ of $A'$, and $[K:\Qz] = [K':\Qz]$. 
\end{lemma}

\begin{proof}
By the Chebotarev density theorem (\cite{neukirch2013algebraic}, Ch. VII, \S 13, page 545) there is a positive density of rational primes that split completely in $K$, $K'$ and $\Qz(\zeta_l)$ at the same time. As there are only finitely many primes at which $A$ or $A'$ has bad reduction, there exists a rational prime $p \equiv 1 \text{ mod } l$ (this is equivalent to $p$ splitting completely in $\Qz(\zeta_l)$) that splits completely in $K$ and $K'$, such that furthermore $A$ has good reduction at all $\mfp \mid p$, and $A'$ has good reduction at all $\mfq \mid p$.

For such a rational prime $p$, the degree of the polynomial 
\[
\Lsp{1_K}
\]
is equal to $2d\# \{\mfp \mid p\} = 2d[K:\Qz]$. Similarly, as $\psi(1_K)$ is the trivial character of $K'$, the degree of 
\[
\Lspx{\psi(1_K)}
\]
equals $2d'[K':\Qz]$, hence we find 
\begin{equation}\label{eq:d*degree}
d[K:\Qz] = d'[K':\Qz].
\end{equation}

Let $\mfp$ be a prime of $K$ lying over $p$. As $p$ is totally split, $K_{\mfp}$ is isomorphic to $\mathbb{Q}_p$. As $\mathbb{Q}_p$ contains the $(p-1)^\text{th}$ roots of unity, adding a $(p-1)^\text{th}$ root of a uniformizer creates a totally ramified Galois extension with Galois group $\Zz/(p-1)\Zz$. Finally, $l \mid p-1$, thus $\Zz/l\Zz$ is a quotient of $\Zz/(p-1)\Zz$. We conclude that $K_{\mfp}$ has a ramified extension of degree $l$.

It follows by the Grunwald-Wang theorem that there exists a character $\chi \in \Hlab[l]$ that ramifies at all primes lying over $p$ except at a single prime $\tilde{\mfp} \in \cP_{K',p}$. Then
\[
\Lsp{\chi} = L_{\tilde{\mfp}}(A/K, \chi, T),
\]
which is a polynomial of degree $2d$. As $p$ is totally split in $K'$, the degree of $\Lspx{\psi(\chi)}$ is equal to
\[
2d' \cdot \# \{\mfq \mid p: \psi(\chi) \text{ unramified at $\mfq$} \}.
\]
This implies $d'\mid d$. By symmetry of $K$ and $K'$, we conclude that $d = d'$. By (\ref{eq:d*degree}), we also have $[K:\Qz] = [K':\Qz]$.
\end{proof}

\begin{remark}
The equality $L(A/K, s) = L(A'/K', s)$ (without equality of twisted $L$-series) implies that $d[K:\Qz] = d'[K':\Qz]$. The author does not know whether or not it necessarily holds that $d = d'$ and $[K:\Qz] = [K':\Qz]$.
\end{remark}

We state and prove a corollary of Theorem~\ref{theorem:local factors are equal} that is useful for proving Theorem~\ref{theoremA:abelian varieties are isogenous through L-series}.

\begin{corollary}\label{cor:wantedprops}
Keeping the notation of Theorem~\ref{theorem:local factors are equal}, we have the following.
\begin{enumerate}
\item The map $\phi$ is norm-preserving. \label{cor:1}
\item For any $\mfp \in \cS$ and $1 \leq i \leq 2d$ we have $\ap_i = \aphip_i$.
\item  For any character $\chi \in \Hlab[l]$ and any $\mfp \in \cS$ we have that
\[
\psi(\chi)(\phi(\mfp)) = \chi(\mfp).
\]
\label{cor:3}
\end{enumerate}
\end{corollary}

\begin{proof}
\mbox{ }
\begin{enumerate}
\item Let $\mfp \in \cS$. As $A$ has good reduction at $\mfp$, the $L$-series $L_\mfp(A/K, T)$ has degree $2d f_{\mfp}$. Similarly, the degree of $L_{\phi(\mfp)}(A'/K', T)$ equals $2df_{\phi(\mfp)}$. By Lemma~\ref{lemma: same dimension}, $d = d'$, hence $f_{\mfp} = f_{\phi(\mfp)}$, i.e. $\phi$ is norm-preserving.
\item Using the fact that $\mfp$ and $\phi(\mfp)$ have the same norm, by comparing the coefficients of the equality $L_\mfp(A/K, T) = L_{\phi(\mfp)}(A'/K', T)$ we have that $\ap_i = \aphip_i$ for all $\mfp \in \cP_K$ and $1 \leq i \leq 2d$. \label{cor:2}
\item For any character $\chi \in \Hlab[l]$ we have, by inspecting the coefficient of $T^{2d}$ of the equality $L_\mfp(A/K, \chi, T) = L_{\phi(\mfp)}(A'/K', \psi(\chi), T)$, that 
\[
\chi(\mfp)^{2d} \ap_{2d} = \psi(\chi)(\phi(\mfp))^{2d} \aphip_{2d}.
\]
Because $\ap_{2d} = \aphip_{2d} \neq 0$, we find $\chi(\mfp)^{2d} = \psi(\chi)(\phi(\mfp))^{2d}$. As $\chi$ and $\psi(\chi)$ are characters of order $l$ and $(l, 2d) = 1$, it follows that $\chi(\mfp) = \psi(\chi)(\phi(\mfp))$. \qedhere
\end{enumerate}
\end{proof}

For the remainder of the section, fix a rational prime $p \in S$. As we cannot create characters ramified at a general $\mfp$, it is useful to create a connection between characters that are unramified at certain primes, as is done by the following lemma.

\begin{lemma}\label{lem:unram->unram}
Let $\chi \in \Hlab[l]$ be a character unramified at all primes lying over $p$. Then $\psi(\chi)$ is unramified at all primes lying over $p$. 
\end{lemma}

\begin{proof}
The degree of the polynomial $L_p(A/K, \chi, T)$ is $2d[K:\Qz] = 2d'[K':\Qz]$.  The degree of $L_p(A'/K', \psi(\chi), T)$ is equal to
\[
2d'[K':\Qz]\; - \sum_{\substack{\mfq \mid p \\ \psi(\chi) \text{ ramified at } \mfq}} 2d' f_{\mfq},
\]
so we have equality only if $\psi(\chi)$ is unramified at all primes lying over $p$.
\end{proof}

We prove Theorem~\ref{theorem:local factors are equal} using induction: for a certain $f \geq 1$, we focus on the coefficient of $T^f$ of the equality $L_p(A/K, \chi, T) = L_p(A'/K', \psi(\chi), T)$, and derive a bijection of primes
\[
\{\mfp \in \cP_{K, p}: f_{\mfp} \mid f, \, \apf \neq 0\}\DistTo \{\mfq \in \cP_{K', p}: f_{\mfq} \mid f, \, \aqf \neq 0\}
\]
after which we show that these bijections are compatible for different $f$.
Any of the factors $L_{\mfp}(A/K, \chi, T)$ has a non-zero coefficient, namely $\ap_{2d}$. Therefore, the inductive argument results in a bijection $\cP_{K, p} \to \cP_{K', p}$.

Before we continue to the induction hypothesis, we begin with a number of convenient definitions.

\begin{definition}
Let $f \geq 1$ be an integer. We define the following sets (with the convention that $\ap_i = 0$ if $i > 2d$): 
\begin{itemize}
\item $\cQ_{K, <f} := \{\mfp \in \cP_{K, p}: \exists 1 \leq i < f/f_\mfp: \ap_i \neq 0\}$;
\item $\cQ_{K, f} := \{\mfp \in \cP_{K, p}: f_{\mfp} \mid f, \, \apf \neq 0\}$;
\item $\cQ_{K, f}^+ := \{\mfp \in \cQ_{K, f}: \apf > 0\}$;
\item $\cQ_{K, f}^- := \{\mfp \in \cQ_{K, f}: \apf < 0\}$.
\end{itemize}
The sets $\cQ_{K', <f}$, $\cQ_{K', f}$, $\cQ_{K', f}^+$ and $\cQ_{K', f}^-$ are defined similarly. 
\end{definition}

\noindent\textbf{Induction hypothesis.} Let $f \geq 1$ be an integer. Assume that there is a bijection $\phi_{<f}: \cQf \to \cQfx$ such that for any $\mfp \in \cQf$ and any $\chi \in \Hlab[l]$ the equality
\[
L_{\mfp, <f}(A/K, \chi, T) = L_{\phi(\mfp), <f}(A'/K', \psi(\chi), T)
\]
holds. Note that this condition is empty for $f = 1$.

\begin{remark}
In the induction step below, we will only construct a bijection of primes $\cQ_{K, f}^+ \to \cQ_{K', f}^+$. The construction of the bijection $\cQ_{K, f}^- \to \cQ_{K', f}^-$ is analogous; one needs to swap the roles of $\cQ_{K, f}^+$ and $\cQ_{K, f}^-$, ``positive'' for ``negative'' and ``maximize'' for ``minimize''.
\end{remark}

\noindent \textbf{Induction step.} In order to extend the induction hypothesis, we create a bijection of primes $\phi: \cQ_{K, f} \to \cQ_{K',f}$ such that 
for any $\mfp \in \cQ_{K ,f}$ we have 
\[
\chi(\mfp)^{f/f_{\mfp}} \apf = \psi(\chi)(\phi(\mfp))^{f/f_{\phi(\mfp)}} \aphipf.
\]
This is done in Lemma~\ref{lemma: there is a phi with properties}. We then show that $\phi$ and $\phi_{<f}$ agree on all primes in $\cQ_{K, f}$ at which both are defined (Lemma~\ref{lemma:phi is consistent}). These two properties combined allow us to conclude that (see Lemma~\ref{lemma:induction extended})
\[
L_{\mfp, \leq f}(A/K, \chi, T) = L_{\phi(\mfp), \leq f}(A'/K', \psi(\chi), T).
\]

In line with the proof of Theorem~\ref{theorem: from L-functions of elliptic curves to prime bijection}, we use a character $\chipos \in \Hlab[l]$ that maximizes a certain function. The next couple of lemmas set up this construction.

\begin{lemma}\label{lemma:equalcoefficients}
For any $\chi \in \Hlab[l]$ we have the equality
\[
\sum_{\mfp \in \cQ_{K, f}} \chi(\mfp)^{f/f_{\mfp}} \apf = 
\sum_{\mfq \in \cQ_{K', f}} \psi(\chi)(\mfq)^{f/f_{\mfq}} \aqf.
\]
\end{lemma}

\begin{proof}
We inspect the coefficient of $T^f$ in both $L_p(A/K, \chi, T)$ and $L_p(A'/K', \psi(\chi), T)$. Note that
\begin{align*}
[T^f]L_p(A/K, \chi, T) 	&= [T^f] \prod_{\mfp \in \cP_{K, p}} L_\mfp(A/K, \chi, T) \\
						&= [T^f] \prod_{\mfp \in \cP_{K, p}} L_{\mfp, \leq f}(A/K, \chi, T)\\
						&= [T^f] \prod_{\mfp \in \cP_{K, p}} \Bigg(L_{\mfp, < f}(A/K, \chi, T) + \begin{cases} 0 & \text{if } f_\mfp \nmid f \\ \chi(\mfp)^{f/f_{\mfp}}\apf T^f & \text{if } f_{\mfp} \mid f\end{cases} \Bigg)\\
						&= \sum_{\substack{\mfp \in \cP_{K, p}\\ f_{\mfp} \mid f}} \chi(\mfp)^{f/f_{\mfp}} \apf + [T^f] \prod_{\mfp \in \cP_{K, p}} L_{\mfp, < f}(A/K, \chi, T)\\
						&= \sum_{\mfp \in \cQ_{K, f}} \chi(\mfp)^{f/f_{\mfp}} \apf + [T^f] \prod_{\mfp \in \cQf} L_{\mfp, < f}(A/K, \chi, T).
\end{align*}
The second to last equality is obtained by noting that the constant coefficient of $L_\mfp(A/K, \chi, T)$ always equals $1$, whilst the last follows from the fact that $L_{\mfp, < f}(A/K, \chi, T) = 1$ if $\mfp \notin \cQf$, and that $\apf = 0$ if $\mfp \notin \cQ_{K, f}$. Similarly,
\[
[T^f]L_p(A'/K', \psi(\chi), T) = \sum_{\mfq \in \cQ_{K', f}} \chi(\mfq)^{f/f_{\mfq}} \aqf + [T^f] \prod_{\mfq \in \cQfx} L_{\mfq, < f}(A'/K', \psi(\chi), T).
\]
The properties of $\psi$ assert that $L_p(A/K, \chi, T) = L_p(A'/K', \psi(\chi), T)$. By the induction hypothesis, 
\[
\prod_{\mfp \in \cQf} L_{\mfp, < f}(A/K, \chi, T) = \prod_{\mfq \in \cQfx} L_{\mfq, < f}(A'/K', \psi(\chi), T).
\]
We conclude that 
\[
\sum_{\mfp \in \cQ_{K, f}} \chi(\mfp)^{f/f_{\mfp}} \apf = 
\sum_{\mfq \in \cQ_{K', f}} \psi(\chi)(\mfq)^{f/f_{\mfq}} \aqf. \qedhere
\]
\end{proof}

\begin{corollary}\label{corollary:equalcoefficients}
We have the equality
\[
\sumf \apf = \sumfx \aqf.
\]
\end{corollary}

\begin{proof}
This follows from Lemma~\ref{lemma:equalcoefficients} by taking $\chi = 1_K$ and noting that $\psi(1_K) = 1_{K'}$.
\end{proof}

\begin{definition}\label{definition:the map c}
Let $c: \Hlab[l] \to \Cz$ be the map given by
\[
\chi \mapsto \sumf \big(1 - \chi(\mfp)^{f/f_\mfp}\big) \apf.
\]
Similarly, define $c': \Hlabx[l] \to \Cz$ by
\[
\chi' \mapsto \sumfx \big(1 - \chi'(\mfq)^{f/f_\mfq}\big) \aqf.
\]
\end{definition}

\begin{lemma}\label{lemma:c is equal}
For any $\chi \in \Hlab[l]$ we have the equality
\[
c(\chi) = c'(\psi(\chi)).
\]
\end{lemma}

\begin{proof}
Subtract the equalities from Lemma~\ref{lemma:equalcoefficients} and Corollary~\ref{corollary:equalcoefficients}.
\end{proof}

The equality of Lemma~\ref{lemma:c is equal} allows us to derive properties of $\psi(\chi)$ for characters at which the real part of $c$ is maximal (over all characters of order $l$), which will be made precise in Lemmas~\ref{lemma: property of maximiser} and~\ref{lemma:properties psichipos}.

\begin{definition}
Define the following $l^\text{th}$ roots of unity:
\begin{align*}
Z &:= \exp(2 \pi i / l);\\
\zeta &:=  \exp(\pi i (l-1) / l).
\end{align*}
Moreover, let 
\begin{align*}
P &:= \text{Re}(Z);\\
\rho &:=  \text{Re}(\zeta).
\end{align*}
\end{definition}

\begin{remark}\label{remark: properties of Zz}
The values $Z$ and $\zeta$ have the following properties.
\begin{enumerate}
\item $Z$ has the largest real part of all $l^{\text{th}}$ roots of unity other than $1$ (along with $Z^{-1}$). \label{enum: Z max}
\item $\zeta$ has the smallest real part of all $l^{\text{th}}$ roots of unity (along with $\zeta^{-1}$). \label{enum: z min}
\item The equality $\zeta Z = \zeta^{-1}$ holds. \label{enum: zZ = z^-1}
\end{enumerate}
\end{remark}

\begin{lemma}\label{lemma: property of maximiser}
Any $\chi \in \Hlab[l]$ that maximises the real part of $c$ has the following properties for any $\mfp \in \cQ_{K, f}$:
\begin{itemize}
\item If $\mfp \in \cQ_{K, f}^+$, then $\chi(\mfp)^{f/f_{\mfp}} = \zeta^{\pm 1}$.
\item If $\mfp \in \cQ_{K, f}^-$, then $\chi(\mfp)^{f/f_{\mfp}} = 1$.
\end{itemize}
Conversely, any character with these properties maximises the real part of $c$. The same holds if we replace $K$ with $K'$, $c$ with $c'$ and $\mfp$ with a prime $\mfq$ of $K'$. 
\end{lemma}

\begin{proof}
For any $\mfp \in \cQ_{K, f}$ we have $f/f_{\mfp} \leq 2d$. As a result, $l$ and $f/f_{\mfp}$ are coprime, hence $\zeta^{\pm 1}$ is an ${f/f_{\mfp}}^\text{th}$ power of an $l^\text{th}$ root of unity. By Lemma~\ref{lemma:grunwald-wang in any order}, there exist characters with the properties stated in the lemma. Furthermore, each individual term of the sum
\[
\text{Re}\big(c(\chi)\big) = \sumf \Big(1 - \text{Re}\big(\chi(\mfp)^{f/f_{\mfp}}\big) \Big) \apf
\]
has upper bound $0$ if $\mfp \in \cQ_{L, f}^-$ and upper bound $(1 - \rho) \apf$ if $\mfp \in \cQ_{K, f}^+$. Any character with the aforementioned properties meets these upper bounds, hence maximises the real part of $c$. Conversely, the upper bound for the term at $\mfp$ is met exactly when 
\[
\text{Re}\big(\chi(\mfp)^{f/f_{\mfp}}\big) = 
\begin{cases}
\rho & \text{ if } \mfp \in \cQ_{K, f}^+;\\
1 & \text{ if } \mfp \in \cQ_{K, f}^-.
\end{cases}
\]
The only two $l^{\text{th}}$ roots of unity with real part equal to $\rho$ are $\zeta$ and $\zeta^{-1}$, hence any character that maximises the real part of $c$ must have the stated properties.
\end{proof}

%Define $Z_{\mfp}$ to be a $l^\text{th}$ root of unity such that 
%\[
%\text{Re}(Z_{\mfp}^{f/f_{\mfp}}) = \max \{\text{Re}(\zeta^{f/f_{\mfp}}): \zeta^l = 1, \zeta^{f/f_{\mfp} \neq 1} \}.
%\]
%Furthermore, $\zeta_{\mfp} = \min \{\text{Re}(\zeta^{f/f_{\mfp}}): \zeta^l = 1\}$. We define $\zeta_{M, \mfq}$ and $\zeta_{m, \mfq}$, where $\mfq \mid p$ is a prime of $K'$ such that $f_{\mfq} \mid f$, similarly.

\begin{definition}\label{definition:chipos}
Define $\chipos \in \Hlab[l]$ to be any character with the following properties.
\begin{enumerate}
\item If $\mfp \in \cQ_{K, f}^+$, then $\chipos(\mfp)^{f/f_{\mfp}} = \zeta$.
\item If $\mfp \in \cQ_{K, f}^-$, then $\chipos(\mfp) = 1$.
\end{enumerate}
Such a character exists by Lemma~\ref{lemma:grunwald-wang in any order}. Note that it maximises the real part of $c$ over all $\chi \in \Hlab[l]$. 
\end{definition}

\begin{lemma}\label{lemma:properties psichipos}
The character $\psi(\chipos)$ is unramified at all primes of $K'$ lying over $p$ and furthermore it maximises the real part of $c'$. 
\end{lemma}

\begin{proof}
By Lemma~\ref{lem:unram->unram}, $\psi(\chipos)$ is unramified at all primes lying over $p$. 
Suppose the character $\chi' \in \Hlabx[l]$ maximises the real part of $c'$. By definition $\chipos$ maximises the real part of $c$, hence
\[
c(\chipos) \geq c(\psi^{-1}(\chi')).
\]
By Lemma~\ref{lemma:c is equal}, this is equivalent to
\[
c'(\psi(\chipos)) \geq c'(\chi').
\]
By assumption on $\chi'$, $\psi(\chipos)$ maximises the real part of $c'$. 
\end{proof}

\begin{corollary}\label{corollary:values_of_psichi}
The character $\psi(\chipos)$ has the following properties.
\begin{itemize}
\item For any $\mfq \in \cQ_{K', f}^+$ we have $\psi(\chipos)(\mfq)^{f/f_\mfq} = \zeta^{\pm 1}$.
\item For any $\mfq \in \cQ_{K', f}^-$ we have $\psi(\chipos)(\mfq)^{f/f_\mfq} = 1$.
\end{itemize}
\end{corollary}

\begin{proof}
This is a consequence of Lemma~\ref{lemma: property of maximiser} and Lemma~\ref{lemma:properties psichipos}.
\end{proof}

\begin{definition}\label{definition: chimfp}
For any prime $\mfp \in \cQ_{K, f}^+$ fix a character $\chi_\mfp$ such that for any prime $\tilde{\mfp} \in \cQ_{K, f}$ we have
\[
\chi_\mfp(\tilde{\mfp})^{f/f_{\tilde{\mfp}}} = \begin{cases} 
Z & \text{if } \tilde{\mfp} = \mfp;\\
1 & \text{otherwise.}
\end{cases}
\]
The existence of $\chi_\mfp$ is guaranteed by Lemma~\ref{lemma:grunwald-wang in any order}.
\end{definition}

\begin{lemma}\label{lemma:psichi has value 1 on negative primes}
Let $\mfq \in \cQ_{K', f}^-$. Then $\psi(\chi_{\mfp})(\mfq)^{f/f_{\mfq}} = 1$. 
\end{lemma}

\begin{proof}
For any prime $\tilde{\mfp} \in \cQ_{K, f}$ we can calculate the value of $\chi_\mfp \chipos$ at $\tilde{\mfp}$:
\[
(\chi_\mfp \chipos)(\tilde{\mfp})^{f/f_{\tilde{\mfp}}} = \begin{cases} 
\zeta^{-1} & \text{if } \tilde{\mfp} = \mfp;\\
\zeta & \text{if } \tilde{\mfp} \neq \mfp \text{ and } \tilde{\mfp} \in \cQ_{K, f}^+;\\
1 & \text{if } \tilde{\mfp} \in \cQ_{K, f}^-.
\end{cases}
\]
By Lemma~\ref{lemma: property of maximiser} $\chi_\mfp \chipos$ maximises the real part of $c$; hence $\psi(\chi_\mfp \chipos) = \psi(\chi_\mfp) \psi(\chipos)$ maximises the real part of $c'$. Therefore
\[
\psi(\chi_\mfp \chipos)(\mfq)^{f/f_\mfq} = \begin{cases} 
\zeta^{\pm 1} & \text{if } \aqf > 0;\\
1 & \text{if } \aqf < 0.
\end{cases}
\]
We already know the value of $\psi(\chipos)(\mfq)^{f/f_\mfq}$ from Corollary~\ref{corollary:values_of_psichi}, hence it is immediate that $\psi(\chi_\mfp)(\mfq)^{f/f_\mfq} = 1$ if $\aqf < 0$.
\end{proof}

\begin{lemma}\label{lemma: there is a phi with properties}
There exists a bijection $\phi: \cQ_{K, f}^+ \to \cQ_{K', f}^+$ such that for any $\mfp \in \cQ_{K, f}$ we have $\apf = \aphipf$ and any $\chi \in \Hlab[l]$ we have
\[
\chi(\mfp)^{f/f_{\mfp}} = \psi(\chi)(\phi(\mfp))^{f/f_{\phi(\mfp)}}.
\]
\end{lemma}

The proof of this lemma consists of multiple parts: first we prove that $\psi(\chi_{\mfp})$ has value $1$ on all primes $\mfq \in \cQ_{K', f}$ except one, using an inductive argument. From this we derive a map $\cQ_{K, f} \to \cQ_{K', f}$  and show that it has the required properties. 

\begin{lemma}\label{lemma: there is a prime map}
For any $\mfp \in \cQ_{K, f}^+$ there exists a prime $\phi(\mfp) \in \cQ_{K', f}^+$ such that 
\[
\apf = \aphipf.
\]
Moreover, for any $\chi \in \Hlab[l]$ we have $\chi(\mfp)^{f/f_{\mfp}} = \psi(\chi)(\phi(\mfp))^{f/f_{\phi(\mfp)}}$.
\end{lemma}

\begin{proof}
Denote by $\mfp_1, \dots, \mfp_s$ the primes in $\cQ_{K, f}^+$, sorted such that 
\[
\big(a_{\mfp_1}\big)_{f/f_{\mfp_1}} \leq \big(a_{\mfp_2}\big)_{f/f_{\mfp_2}} \leq \dots \leq \big(a_{\mfp_s}\big)_{f/f_{\mfp_s}}.
\]  
We give a proof by induction. \\

\noindent\textbf{Induction hypothesis.} Suppose that for some $i \geq 1$ the following are true.
\begin{itemize}
\item For every $j < i$ there is a prime $\mfq_j$ of $K'$ lying over $p$ such that $\apfj = \aqfj$, and $\mfq_j \neq \mfq_k$ if $j \neq k$.
\item The equality $\chi(\mfp_j)^{f/f_{\mfp_j}} = \psi(\chi)(\mfq_j)^{f/f_{\mfq_j}}$ holds for all $j < i$ and all $\chi \in \Hlab[l]$ unramified at all primes lying over $\mfp$. 
\item We have 
\[
\big(a'_{\mfq_1}\big)_{f/f_{\mfq_1}} \leq \big(a'_{\mfq_2}\big)_{f/f_{\mfq_2}} \leq \dots \leq \big(a'_{\mfq_{i-1}}\big)_{f/f_{\mfq_{i-1}}} \leq \min \big\{\aqf: \mfq \in \cQ_{K', f}^+\backslash \{ \mfq_1, \dots, \mfq_{i-1}\}\big\}.
\]
\end{itemize}
Note that the hypothesis is empty for $i = 1$.\\

\noindent \textbf{Induction step.}
Without loss of generality we assume that (one can swap the roles of $K$ and $K'$ if necessary):
\[
\apfi \leq \min \big\{\aqf: \mfq \in \cQ_{K', f}^+\backslash \{ \mfq_1, \dots, \mfq_{i-1}\}\big\}
\] 

Let $\chi_i$ be shorthand for $\chi_{\mfp_i}$ as in Definition~\ref{definition: chimfp}. Note that $c(\chi_i) = (1 - Z) \apfi$. \\

\noindent The character $\psi(\chi_i)$ has value $1$ on all primes $\mfq_1, \dots, \mfq_{i-1}$ by the induction hypothesis. By Lemma~\ref{lemma:psichi has value 1 on negative primes}, for any prime $\mfq \in \cQ_{K', f}^-$ we have $\psi(\chi_i)(\mfq)^{f/f_{\mfq}} = 1$.

We use this to find a lower bound on $c'(\psi(\chi_i))$. Let $t$ be the cardinality of the set $\{\mfq \in \cQ_{K', f}^+: \psi(\chi_i)(\mfq)^{f/f_{\mfq}} \neq 1\}$. Then
\[
\text{Re}\big(c'(\psi(\chi_i))\big) = \sumfpx \Big(1 - \text{Re}\big(\psi(\chi_i)(\mfq)^{f/f_{\mfq}}\big)\Big) \aqf \geq t \cdot (1 - P)\apfi,
\]
as $\psi(\chi_i)(\mfq)^{f/f_{\mfq}} \neq 1$ implies that 
\[
\Big(1 - \text{Re}\big(\psi(\chi_i)(\mfq)^{f/f_{\mfq}}\big)\Big) \aqf \geq (1 - P) \aqf \geq (1 - P)\apfi.
\]

By Lemma~\ref{lemma:c is equal} we have $\text{Re}\big(c'(\psi(\chi_i)\big) = \text{Re}\big(c(\chi_i)\big) = (1 - P)\apfi$, hence $t = 1$. Hence there is a unique prime $\mfq_i \in \cQ_{K', f}^+$ such that $\psi(\chi_i)(\mfq_i)^{f/f_{\mfq_i}} \neq 1$. It follows that $c'(\psi(\chi_i)) = \Big(1 - \psi(\chi_i)(\mfq_i)^{f/f_{\mfq_i}}\Big) \aqfi$.

\noindent From $c(\chi_i) = c'(\psi(\chi_i))$ we find 
\[
(1 - Z)\apfi = (1 - \psi(\chi_i)(\mfq_i)^{f/f_{\mfq_i}}) \aqfi.
\]
We claim and prove that this implies that 
\[
\psi(\chi_i)(\mfq_i)^{f/f_{\mfq_i}} = Z = \chi_i(\mfp_i)^{f/f_{\mfp_i}} \text{ \,\,\,and\,\,\, } \apfi = \aqfi.
\]

\noindent Both $1 - Z$ and $1 - \psi(\chi_i)(\mfq_i)^{f/f_{\mfq_i}}$ lie on a circle with center point $1$ and radius $1$. 

As $\apfi$ and $\aqfi$ are real numbers unequal to zero, $1 - Z$ and $1 - \psi(\chi_i)(\mfq_i)^{f/f_{\mfq_i}}$ lie on a line through the origin. However, this line and circle intersect in only two points, one of which is the origin.

Because $1 - Z \neq 0 \neq 1 - \psi(\chi_i)(\mfq_i)^{f/f_{\mfq_i}}$, both must be equal to the second intersection point, hence $1 - Z = 1 - \psi(\chi_i)(\mfq_i)^{f/f_{\mfq_i}}$. It follows immediately that $\apfi = \aqfi$. \\

\noindent Summarising, we have proven that for any $\mfq \in \cQ_{K',f}$
\begin{equation} \label{equation: properties of psichii}
\psi(\chi_i)(\mfq)^{f/f_{\mfq}} =  \begin{cases}  
    Z & \text{if  } \mfq = \mfq_i; \\
    1 & \text{otherwise,}\\
  \end{cases} 
\end{equation}
and that $\apfi = \aqfi$. Furthermore, note
\[
\aqfi = \apfi \leq \min \big\{\aqf: \mfq \in \cQ_{K', f}^+\backslash \{ \mfq_1, \dots, \mfq_{i-1}\}\big\}.
\]
To complete the induction step, we check that for any $\chi \in \Hlab[l]$ we have 
\[
\chi(\mfp_i)^{f/f_{\mfp_i}} = \psi(\chi)(\mfq_i)^{f/f_{\mfq_i}}.
\] 
For this we use the multiplicativity of $\psi$; in particular, multiplying $\chi$ with $\chi_i$ changes only the value at a single prime in $\cQ_{K, f}$ (namely $\mfp_i$). By what we have just proven, multiplying with $\psi(\chi_i)$ has a similar effect; it changes only the value at a single prime in $\cQ_{K', f}$.\\

\noindent Indeed, for any $\mfp \in \cQ_{K, f}$ we have
\[
(\chi \chi_i)(\mfp)^{f/f_{\mfp}} = 
\begin{cases}
\chi(\mfp)^{f/f_{\mfp}} & \text{ if } \mfp \neq \mfp_i;\\
Z\chi(\mfp)^{f/f_{\mfp}} & \text{ if } \mfp = \mfp_i.
\end{cases}
\]
A quick calculation shows $c(\chi \chi_i) = c(\chi) + (1 - Z)\chi(\mfp_i)^{f/f_{\mfp_i}} \apfi$. \\

\noindent Similarly, by~(\ref{equation: properties of psichii}) and multiplicativity of $\psi$, for any $\mfq \in \cQ_{K',f}$ we have
\[
\psi(\chi \chi_i)(\mfq)^{f/f_{\mfq}} = 
\begin{cases}
\psi(\chi)(\mfq)^{f/f_{\mfq}} & \text{ if } \mfq \neq \mfq_i;\\
Z\chi(\mfq)^{f/f_{\mfq}} & \text{ if } \mfq = \mfq_i.
\end{cases}
\]
Hence we have $c'(\psi(\chi \chi_i)) = c'(\psi(\chi)) + (1 - Z)\psi(\chi)(\mfq_i)^{f/f_{\mfq_i}} \aqfi$.

As $c(\chi \chi_i) = c'(\psi(\chi \chi_i))$ and $c(\chi) = c'(\psi(\chi))$ it follows that
\[
 (1 - Z)\chi(\mfp_i)^{f/f_{\mfp_i}} \apfi = (1 - Z)\psi(\chi)(\mfq_i)^{f/f_{\mfq_i}} \aqfi.
\]
We already proved that $\apfi = \aqfi$, thus $\chi(\mfp_i)^{f/f_{\mfp_i}} = \psi(\chi)(\mfq_i)^{f/f_{\mfq_i}}$.

\noindent This concludes the proof by induction. The lemma follows by setting $\phi(\mfp_i) := \mfq_i$.
\end{proof}

\begin{remark}
As the equality $\chi(\mfp)^{f/f_{\mfp}} = \psi(\chi)(\phi(\mfp))^{f/f_{\phi(\mfp)}}$ holds for any $\chi \in \Hlab[l]$ that is unramified at all primes lying over $p$, the map $\phi$ is well-defined and independent of the choice of $\chi_i$. 
\end{remark}

\begin{lemma}
The map $\phi$, defined in the previous lemma, is a bijection $\cQ_{K, f}^+ \to \cQ_{K', f}^+$. 
\end{lemma}

\begin{proof}
Note that it is injective: if $j \neq i$, then we have
\[
\psi(\chi_i)(\phi(\mfp_i))^{f/f_{\phi(\mfp_i)}} = Z,
\]
while $\psi(\chi_i)(\phi(\mfp_j))^{f/f_{\phi(\mfp_j)}} = 1$. By symmetry of $K$ and $K'$, $\phi$ is a bijection.
\end{proof}

\begin{lemma}\label{lemma:phi is consistent}
The map $\phi$ extends $\phi_{<f}$ to a bijection $\phi_{\leq f}: \cQ_{K, \leq f} \to \cQ_{K', \leq f}$. 
\end{lemma}

\begin{proof}
As $\phi$ is a bijection $\cQ_{K, f} \to \cQ_{K', f}$ and $\phi_{<f}$ is a bijection $\cQ_{K, < f} \to \cQ_{K', < f}$ it suffices to prove that for any $\mfp \in \cQ_{K, <f} \cap \cQ_{K, f}$ we have $\phi_{<f}(\mfp) = \phi(\mfp)$.

We argue by contradiction. Suppose there is a prime  $\mfp \in \cQ_{K, <f} \cap \cQ_{K, f}$ for which $\phi_{<f}(\mfp) \neq \phi(\mfp)$. For convenience, denote $\mfq_1:= \phi_{<f}(\mfp)$ and $\mfq_2:= \phi(\mfp)$. As $\cQ_{K, <f} = \bigcup_{1 \leq j \leq f} \cQ_{K, j}$, there is a $1\leq j < f$ such that $\mfp \in \cQ_{K, j}$.

By Lemma~\ref{lemma:grunwald-wang in any order} there exists a character $\chi' \in \Hlabx[l]$ such that 
\begin{align*}
\chi'(\mfq_1) &= 1 \text{ and }\\
\chi'(\mfq_2) &= Z.
\end{align*}
It follows from the induction hypothesis that $\psi^{-1}(\chi')(\mfp)^{j/f_{\mfp}} = 1$. This, along with the fact that $\psi^{-1}(\chi')(\mfp)$ is an $l^\text{th}$ root of unity and $j/f_{\mfp} < l$, implies that $\psi^{-1}(\chi')(\mfp) = 1$. Moreover, by  Lemma~\ref{lemma: there is a phi with properties} and $f/f_{\mfq_2} < l$ we have $\psi^{-1}(\chi')(\mfp)^{f/f_{\mfp}} = Z^{f/f_{\mfq_2}} \neq 1$. This is a contradiction.
\end{proof}

\begin{lemma}\label{lemma:induction extended}
For any $\chi \in \Hlab[l]$ and $\mfp \in \cQ_{K, f}^+$ we have 
\[
L_{\mfp, \leq f}(A/K, \chi, T) = L_{\phi(\mfp), \leq f}(A'/K', \psi(\chi), T).
\]
\end{lemma}

\begin{proof}
By the induction hypothesis we have 
\[
L_{\mfp, < f}(A/K, \chi, T) = L_{\phi(\mfp), < f}(A'/K', \psi(\chi), T).
\]
We therefore need only be concerned with the coefficient of $T^f$. The coefficient of $T^f$ in $L_{\mfp, < f}(A, K, \chi, T)$ is equal to 
\[
\chi(\mfp)^{f/f_{\mfp}} \apf,
\]
while the coefficient of $T^f$ in $L_{\phi(\mfp), < f}(A'/K', \psi(\chi), T)$ equals 
\[
\psi(\chi)(\phi(\mfp))^{f/f_{\phi(\mfp)}} \aphipf.
\]
By Lemma~\ref{lemma: there is a phi with properties}, these coefficients are equal.
\end{proof}

This concludes the proof of Theorem~\ref{theorem:local factors are equal}.\qed

\section{Isomorphisms of number fields with certain bijections of primes}\label{section:prime bijections}
This section is devoted to proving Theorems~\ref{theoremA:abelian varieties are isogenous through L-series} and~\ref{theoremA: elliptic curves are isogenous through L-series}. For this we use the main results of Sections~\ref{section:proof non-CM elliptic curves} and~\ref{section:abelian varieties}, namely Theorems~\ref{theorem: from L-functions of elliptic curves to prime bijection} and~\ref{theorem:local factors are equal}. \new{The case of Theorem~\ref{theoremA: elliptic curves are isogenous through L-series} where $\text{End}_{\overline{K}}(E) \not\subseteq K$ requires a separate proof.}

\begin{proposition}
\new{Suppose either the conditions of Theorem~\ref{theoremA:abelian varieties are isogenous through L-series} hold, or those of Theorem~\ref{theoremA: elliptic curves are isogenous through L-series} and additionally $\textup{End}_{\overline{K}}(E) \subseteq K$. Then there is a unique isomorphism $\sigma: K \DistTo K'$ whose induced bijection of primes matches $\phi$.}
\end{proposition}

\begin{proof}
In the case of Theorem~\ref{theoremA:abelian varieties are isogenous through L-series}, we have by Theorem~\ref{theorem:local factors are equal} and Corollary~\ref{cor:wantedprops} an injective norm-preserving map $\phi: \cS \to \cP_{K'}$ such that
\[
\chi(\mfp) = \psi(\chi)(\phi(\mfp))
\]
for any $\mfp \in \cS$ and any character $\chi \in \Hlab[l]$, where $\cS$ consists of all primes of sufficiently high norm. In particular, $\cS$ has density one. 

In the case of Theorem~\ref{theoremA: elliptic curves are isogenous through L-series}, Theorem~\ref{theorem: from L-functions of elliptic curves to prime bijection} guarantees that for $\mfp \in \cP_K$ of sufficiently high norm with $a_\mfp \neq 0$ we have 
\[
\chi(\mfp) = \psi(\chi)(\phi(\mfp)).
\]
The condition that $\text{End}_{\overline{K}}(E) \subseteq K$ is equivalent to the condition that $E/K$ either has no complex multiplication, or complex multiplication by a quadratic extension contained in $K$. We consider both cases.

If $E$ does not have complex multiplication, the set $\{ \mfp \in \cP_{K}: a_{\mfp} \neq 0\}$ has density one by Serre's open image theorem \cite{serreopen}. 

If $E$ does have complex multiplication by a quadratic imaginary field $L$, then Deuring's Criterion \cite{deuring}
%\cite{lang} Chapter 13, \S 4, Theorem 12)
states that if $p$ is unramified in $L$ and $\mfp$ is a prime of $K$ lying over $p$ at which $E$ has good reduction (which holds for all but finitely many $\mfp$), then 
\[
E \text{ has supersinglar reduction at $\mfp$} \Leftrightarrow p \text{ is inert in $L/\Qz$}.
\]
If $L \subseteq K$, then the above equivalence implies that $E$ can only have supersingular reduction at $\mfp$ if $f_{\mfp} \geq 2$, hence the set of such primes has density zero. Because $a_{\mfp} = 0$ implies that $E$ has supersingular reduction at $\mfp$ (see \cite[Ch.\ V, \S 4, Thm.\ 4.1]{silvermanaec}), the set
\[
\{ \mfp \in \cP_{K}: a_{\mfp} \neq 0\}
\]
has density one.

In both cases we have an isomorphism $\psi: \Hkab \to \Hkabx$ with a bijection $\phi$ on a density one subset of the primes. By \cite[Thm.\ 3.1]{GDTHJS}, we know that there exists a unique $\sigma: K \to K'$ whose associated bijection of primes matches $\phi$. 
\end{proof}

\begin{proposition}\label{proposition:CM_case}
\new{Suppose the conditions of Theorem~\ref{theoremA: elliptic curves are isogenous through L-series} hold and additionally $\textup{End}_{\overline{K}}(E) \not\subseteq K$. Call the field of complex multiplication $L$. Then there is a unique isomorphism $\sigma: K \DistTo K'$ whose induced bijection of primes matches $\phi$ on almost all primes in $\textup{Spl}(KL/K)$.}
\end{proposition}

\begin{proof}
\new{Let $\mfp \in \text{Spl}(KL/K)$ be a prime with ramification and inertia degree $1$ at which $E$ has good reduction and denote $p := \mfp \cap \Zz$. As $\mfp$ splits in $KL/K$ and $\mfp$ has inertia degree $1$, $p$ splits in $L/\Qz$. Hence $E$ has ordinary reduction at $\mfp$ and therefore $a_{\mfp} \neq 0$. Theorem~\ref{theorem: from L-functions of elliptic curves to prime bijection} now states that
\[
\chi(\mfp) = \psi(\chi)(\phi(\mfp))
\]
for any such $\mfp$ and any $\chi \in \Hkab[l]$. Now \cite[Theorem B]{GDTHJS} guarantees the result.}
\end{proof}

\begin{proof}[Proofs of Theorem~\ref{theoremA:abelian varieties are isogenous through L-series} and Theorem~\ref{theoremA: elliptic curves are isogenous through L-series}]
From this proposition it follows that if the conditions of Theorem~\ref{theoremA:abelian varieties are isogenous through L-series} hold, then for any $\mfp \in \cS$ we have
\[
L_{\phi(\mfp)}(A^\sigma/K', s) = L_{\sigma(\mfp)}(A^\sigma/K', s) = L_{\mfp}(A/K, s) = L_{\phi(\mfp)}(A'/K', s).
\]
As $\phi(\cS) = \sigma(\cS)$ has density one in the primes of $K'$, Faltings's isogeny theorem implies that $A^\sigma/K'$ and $A'/K'$ are isogenous. Similarly, if the conditions of Theorem~\ref{theoremA: elliptic curves are isogenous through L-series} hold for elliptic curves $E/K$ and $E'/K'$ \new{with $\text{End}_{\overline{K}}(E)\subseteq K$}, then the same reasoning can be used for $E$ and $E'$ instead of $A$ and $A'$ respectively, hence $E^\sigma/K'$ and $E'/K'$ are isogenous. 

\new{Now consider the case where $\text{End}_{\overline{K}}(E)\nsubseteq K$. Proposition~\ref{proposition:CM_case} states that for all primes but a density zero set in $\text{Spl}(KL/K)$ we have
\[
L_{\phi(\mfp)}(E^\sigma/K', s) = L_{\sigma(\mfp)}(E^\sigma/K', s) = L_{\mfp}(E/K, s) = L_{\phi(\mfp)}(E'/K', s).
\]
Now let $\mfp$ be a prime that does not lie over $(2)$ or $(3)$ such that $p := \mfp \cap \Zz$ is inert in $L/\Qz$, and $\mfp$ has inertia and ramification degree $1$. Then $E$ has supersingular reduction at $\mfp$ (as well as any other unramified prime lying over $p$) and by the Hasse bound it follows that $a_{\mfp} = 0$. Hence $a'_{\phi(\mfp)} = 0$ as well and therefore
\[
L_{\phi(\mfp)}(E'/K', s) = 1 + pT^2.
\]
Moreover, as $\sigma^{-1}(\phi(\mfp))$ lies over $p$ and it has inertia and ramification degree $1$, it follows that $a_{\sigma^{-1}(\phi(\mfp))} = 0$, thus
\[
L_{\phi(\mfp)}(E^\sigma/K', s) = L_{\sigma^{-1}(\phi(\mfp))}(E/K, s) = 1 + pT^2.
\]
Hence for almost all primes of inertia degree $1$ we have an equality of local factors. Faltings's isogeny theorem guarantees that $E^\sigma$ and $E'$ are isogenous. }

This concludes the proof of Theorems~\ref{theoremA:abelian varieties are isogenous through L-series} and~\ref{theoremA: elliptic curves are isogenous through L-series}.
\end{proof}

\new{We end this section with a small remark on the difference between the cases $\text{End}_{\overline{K}}(E)\subseteq K$ and $\text{End}_{\overline{K}}(E)\nsubseteq K$. The fundamental difference is that in the first case one can construct an injective norm-preserving map of primes $\phi$ such that}
\[
\psi(\chi)(\phi(\mfp)) = \chi(\mfp)
\]
\new{for all primes $\mfp$ in a set of density one, whilst in the second case this is only guaranteed on a subset of density 1/2.}

\begin{remark}\label{remark:extra_iso}
\new{Let $K = K' = \Qz$, and let $E/\Qz$ be any elliptic curve with complex multiplication by $\Qz(i)$. Let $d \in \Qz^\times / (\Qz^\times)^2$, $p$ any prime congruent to $1 \text{mod} 4$, and let $\chi_{\sqrt{d}}$ be the character associated to the extension $\Qz(\sqrt{d})/\Qz$. Consider the isomorphism $\psi$ defined as follows (see also \cite[Remark 6.17]{GDTHJS}):}
\begin{align*}
\psi: \widecheck{G}_{\Qz}[2] &\to \widecheck{G}_{\Qz}[2] \\
\chi_{\sqrt{d}} &\mapsto  
\begin{cases}
\chi_{\sqrt{d}} & \text{ if } \left(\frac{d \cdot |d|_p}{p} \right) = 1; \\
\chi_{\sqrt{-d}}& \text{ if } \left(\frac{d \cdot |d|_p}{p} \right) = -1.\\
\end{cases}
\end{align*}
\new{This map is an isomorphism that abides $\psi(\chi)(q) = \chi(q)$ for all primes $q$ congruent to $1 \text{ mod } 4$. As a result, we have for any such prime that}
\[
L_q(E/\Qz, \chi, s) = L_q(E/\Qz, \psi(\chi), s).
\]
\new{For primes $q$ congruent to $3 \text{ mod } 4$ (aside from possibly $q = 3$) we have $a_q = 0$, hence}
\[
L_q(E/\Qz, \chi, s) = 1 + qT^2 = L_q(E/\Qz, \psi(\chi), s).
\]
\new{Therefore $\psi$ meets the conditions of Theorem~\ref{theoremA: elliptic curves are isogenous through L-series}. However, it is not the identity, thus it is not induced by an isomorphism $\Qz \DistTo \Qz$.}

\new{In the case of $\text{End}_{\overline{K}}(E)\subseteq K$ one can use \cite[Theorem A]{GDTHJS} to show that all $\psi$ that meet the conditions of Theorem~\ref{theoremA: elliptic curves are isogenous through L-series} are induced by isomorphisms $K \DistTo K'$.}
\end{remark}

\section{Acknowledgements}
I would like to acknowledge the helpful suggestions and remarks from my supervisor Gunther Cornelissen, as well as the pleasant discussions we have had.

\end{document}